\documentclass[matann2]{svjour}

\usepackage{mathptmx}
\usepackage{amsmath}
\usepackage{amsfonts}
\usepackage{latexsym}
\usepackage{mathrsfs}
\usepackage{times}
\usepackage{array}
\usepackage{amscd}
\usepackage{a4}
\usepackage[mathcal]{euscript}
\usepackage{amssymb}
\usepackage{epsf, epic, eepic} 
\usepackage{pstricks}
\usepackage{amsrefs} 
\usepackage{theorem}

\input xy 
\xyoption{all} 
\newcommand{\Kh}{\mathit{Kh}}
\newcommand{\co}{\colon}
\newcommand{\ra}{\rightarrow}
\newcommand{\wt}{\widetilde}
\newcommand{\al}{\alpha}
\newcommand{\ETspace}{\mathbf{X}_\bullet}
\newcommand{\ETsp}[1]{X_{#1}}
\newcommand{\Khsp}[1]{\mathcal{X}^{(#1)}_{\mathit{Kh}}}
\newcommand{\ETQ}{\mathbf{P}}
\newcommand{\C}{\mathsf{C}}
\newcommand{\AbGrp}{\mathsf{Ab}}
\newcommand{\AbTop}{\mathsf{AbTop}}
\newcommand{\Top}{\mathsf{Top}_*}
\newcommand{\vect}{\overline}
\renewcommand{\th}{^{\text{th}}}
\newcommand{\KhSpace}{\mathcal{X}_\mathit{Kh}}
\newcommand{\invlim}{\varprojlim}

\DeclareMathOperator{\Sym}{Sym}
\DeclareMathOperator{\Hom}{Hom}
\DeclareMathOperator{\Id}{Id}
\DeclareMathOperator{\holim}{holim}

\title{Khovanov homotopy types and the Dold-Thom functor}

\author{Brent Everitt, Robert Lipshitz, Sucharit Sarkar and Paul
  Turner\thanks{RL was supported 
by an NSF grant number DMS-0905796 and a Sloan Research Fellowship. SS
was supported 
by a Clay Foundation Postdoctoral Fellowship. PT was partially supported by Swiss National Science Foundation
grant 200021-131967.}}

\institute{
{\sc Brent Everitt:} Department of Mathematics, University of York, York
YO10 5DD, United Kingdom. \email{brent.everitt@york.ac.uk}. 
{\sc Robert Lipshitz, Sucharit Sarkar:} Department of Mathematics,
Columbia University, New York
NY 10027, USA. \email{lipshitz@math.columbia.edu},
\email{sucharit@math.columbia.edu}.
\hspace{1em}{\sc Paul Turner:} Section de math\'ematiques,  
Universit\'e de Gen\`eve, 2-4 rue du Li\`evre, CH-1211, Geneva and D\'epartement de math\'ematiques, 
Universit\'e de Fribourg, Chemin du mus\'ee, CH-1700 Fribourg, Switzerland.
\email{prt.maths@gmail.com}.
}

\titlerunning{Khovanov homotopy types and the Dold-Thom functor}
\authorrunning{Everitt, Lipshitz, Sarkar and Turner}

\begin{document}

\maketitle

%%%%%%%%%%%%%%%%%%%%%%%%%%%%%%%%%%%%%%%%%%%%%%%%%%%%%%%%%%%%%%%%%%%%%%%%%

%\input{configuration}

%\title{}

%\author{Brent Everitt}

%\author{Robert Lipshitz}
%\thanks{RL was supported by an NSF grant number DMS-0905796 and a Sloan Research Fellowship.}
%\email{\href{mailto:lipshitz@math.columbia.edu}{lipshitz@math.columbia.edu}}
%\address{Department of Mathematics, Columbia University, New York, NY 10027}

%\author{Sucharit Sarkar}
%\thanks{SS was supported by a Clay Foundation Postdoctoral Fellowship}
%\email{\href{mailto:sucharit@math.columbia.edu}{sucharit@math.columbia.edu}}
%\address{Department of Mathematics, Columbia University, New York, NY 10027}

%\author{Paul Turner}

% %\subjclass[2010]{\href{http://www.ams.org/mathscinet/search/mscdoc.html?code=57M25,55P42}{57M25,
% %    55P42}}

% \keywords{}

%\date{\today}

\begin{abstract}
We show that the spectrum constructed by Everitt and Turner as a possible Khovanov homotopy type is a product of Eilenberg-MacLane spaces
 and is thus determined by Khovanov homology. By using the Dold-Thom
 functor it can therefore be obtained from the Khovanov homotopy type constructed by Lipshitz and Sarkar.
\end{abstract}

\maketitle

A \emph{Khovanov homotopy type} is a way of associating a (stable)
space to each link $L$ so that the classical invariants
of the space yield the Khovanov homology of $L$. There are two recent
constructions of Khovanov homotopy types, using different techniques
and giving different results \cite{ET-kh-spectrum, RS-khovanov}.
In \cite{ET-kh-spectrum} homotopy
limits were employed to build an $\Omega$-spectrum $\ETspace L
=\{\ETsp{k}(L)\}$ with the following properties:
\begin{description}
\item[(i).]  the  homotopy type is a link invariant, and
\item[(ii).]  the homotopy groups are Khovanov homology: $$\pi_i(\ETspace(L))=\Kh^{-i}(L).$$
\end{description}

The main goal of this note is to prove the following result.

\begin{theorem}\label{thm:product}
  Each of the spaces $\ETsp{k}(L)$ is homotopy equivalent to a product of
  Eilenberg-MacLane spaces.
\end{theorem}

In~\cite{RS-khovanov} the programme of Cohen, Jones
and Segal~\cite{Cohen_Jones_Segal95} was generalized to produce a suspension spectrum
$\KhSpace(L)$ with the following properties:
\begin{description}
\item[(i).] the  homotopy type is a link invariant, and
\item[(ii).]  the reduced cohomology is Khovanov homology: $$\wt{H}^i(\KhSpace(L))=\Kh^i(L).$$
\end{description}
\noindent As a corollary we obtain that $\ETspace(L)$ is homotopy
equivalent to the infinite symmetric product of $\KhSpace(L)$.

To prove Theorem~\ref{thm:product} we use the explicit model, due
to McCord~\cite{McC-top-classsp}, of the
Eilenberg-MacLane spaces.
Given a monoid $G$ and a based topological
space $X$, let $B(G,X)$ denote the set of maps $u\co X\to G$ such that
$u(x)=0$ for all but finitely many $x\in X$. Then $B(G,X)$ is a
monoid, and if $G$ is a group (the case of interest) then $B(G,X)$ is
a group. Moreover, when $G$ is an abelian topological group 
the set $B(G,X)$ can be topologized in a natural way so that
the group operation is continuous.
% Given an abelian monoid $G$ and a based topological
% space $X$, let $B(G,X)$ denote the set of maps $u\co X\to G$ such that
% $u(x)=0$ for all but finitely many $x\in X$. Then $B(G,X)$ is an
% abelian monoid, and if $G$ is a group (the case of interest) then $B(G,X)$ is
% a group. The set $B(G,X)$ can be topologized in a natural way so that
% the group operation is continuous.
This construction has
nice functoriality: 
letting $\AbGrp, \Top$ and $\AbTop$ denote respectively the categories
of 
abelian groups, based topological spaces and topological abelian groups, one has the following result.

\begin{proposition}\label{prop:bifunctor}
  \cite[Proposition 6.7]{McC-top-classsp} McCord's construction is a bifunctor
$$
B(-,-)\co \AbGrp \times \Top \rightarrow \AbTop .
$$
\end{proposition}

Furthermore, as  special case  of~\cite[Theorem
11.4]{McC-top-classsp}, for an abelian group $G$ the space $B(G,S^n)$ is the
Eilenberg-MacLane space $K(G,n)$. 
Thus we may take as \emph{the} Eilenberg-MacLane space functor:
$$
B(-,S^n)\co \AbGrp \rightarrow \AbTop .
$$

Conversely, the following is~\cite[Corollary 4K.7,
p.\ 483]{Hatcher-top-book} (apparently originally due to
Moore; cf.~\cite[p.\ 295]{McC-top-classsp}):
\begin{proposition}\label{prop:TAG-is-K-pi-n}
  A path-connected, commutative topological monoid is a product of
  Eilenberg-MacLane spaces.
\end{proposition}

The spaces $\ETsp{k}(L)$ are built as homotopy limits of diagrams of
spaces. Recall that given a small category $\C$ and a (covariant)
functor $D\co \C \rightarrow \Top$  (a diagram), that $\holim_\C D$ is constructed
as follows (see,
  e.g.,~\cite[Section 11.5]{BK-top-book} or the concise
  notes~\cite[Section 3.7]{SG-top-holim}).  Consider the product
  \begin{equation}\label{eq:product}
 \prod_{\sigma\in N(\C)}\Hom(\Delta^n,D(c_n))=
    \prod_{n\geq 0}
\prod_{\substack{c_0\stackrel{\al_1}{\rightarrow} \dots
    \stackrel{\al_n}{\rightarrow} c_n\\ \alpha_i\neq\Id}}
\Hom(\Delta^n,D(c_n))
  \end{equation}
  where $N(\C)$ is the subset of the nerve of $\C$ consisting of all
  sequences of composable morphisms
  $\sigma=(c_0\stackrel{\alpha_1}{\longrightarrow}
  c_1\stackrel{\alpha_2}{\longrightarrow}\cdots\stackrel{\alpha_n}{\longrightarrow}c_n)$
  in which none of the morphisms are identity maps,
  % where $N(\C)$ is the nerve of $\C$, the
  % $\sigma=(c_0\stackrel{\alpha_1}{\longrightarrow}
  % c_1\stackrel{\alpha_2}{\longrightarrow}\cdots\stackrel{\alpha_n}{\longrightarrow}c_n)$
  % are sequences of composable morphisms in which none of the
  % morphisms are identity maps,
  and $\Hom$ denotes the space of continuous maps from the standard
  $n$-simplex. The homotopy limit $\holim_\C D$ is the subspace of
  this product consisting of those tuples $(f_\sigma)_{\sigma\in
    N(\C)}$ such that the following diagrams commute:
  \begin{equation}\label{eq:sub-1}
    \begin{split} %Hack to get equation number in the right place.
      \xymatrix{
        \Delta^{n-1}\ar[d]_{d^i}\ar[rr]^{{  f_{d_i\sigma}}} & &D(c_n)\ar[d]^{\Id}\\
        \Delta^n\ar[rr]^{{  f_{\sigma}}} & & D(c_n)
      }
    \end{split}
  \end{equation}
  for each $0<i<n$, and
  \begin{equation}\label{eq:sub-2}
    \begin{split} %Hack to get equation number in the right place.
  \xymatrix{
    \Delta^{n-1}\ar[d]_{d^0}\ar[rr]^{{  f_{d_0\sigma}}} & & D(c_n)\ar[d]^{\Id}\\
    \Delta^n\ar[rr]^{{  f_{\sigma}}} & & D(c_n)
  }
  \qquad\text{and}\qquad
  \xymatrix{
    \Delta^{n-1}\ar[d]_{d^n}\ar[rr]^{{  f_{d_n\sigma}}}& &D(c_{n-1})\ar[d]^{D(\alpha_n)}\\
    \Delta^n\ar[rr]^{{  f_{\sigma}}} & & D(c_n)
  }
  \end{split}
  \end{equation}
  corresponding to the cases $i=0$ and $i=n$, respectively.
  Here the map $d^i$ denotes the $i\th$ face inclusion, { 
$d_i\sigma=(c_0\stackrel{\alpha_1}{\longrightarrow}
 \cdots c_{i-1}\stackrel{\alpha_{i+1}\alpha_i}{\longrightarrow}
c_{i+1}\cdots\stackrel{\alpha_n}{\longrightarrow}c_n
  )$ when $0<i<n$, and $d_0,d_n$ similarly.}

The
following is well-known, but for completeness we give its (short)
proof.

\begin{proposition}\label{prop:holim}
  Let $D\co \C \rightarrow \Top$ be a diagram of topological abelian groups and
  continuous group homomorphisms. Then the homotopy limit of $D$ is a
  topological abelian group.
\end{proposition}

\begin{proof}
 Pointwise addition makes the set $\Hom(\Delta^n,D(c_n))$ into an
  abelian group, and the product in formula~(\ref{eq:product}) is the product (topological
  abelian) group. It remains to see that the diagrams~(\ref{eq:sub-1})
  and~(\ref{eq:sub-2}) describe a subgroup of this product. 
Suppose that tuples $(f_{\sigma})$ and $(g_{\sigma})$
  make these diagrams commute. Then the first two diagrams
  automatically commute for the pointwise sum
 {  $(f_{\sigma}+g_{\sigma})$}. The third diagram for the
  pointwise sum becomes,
$$
\xymatrix{
    \Delta^{n-1}\ar[d]_{d^n}\ar[r] &  
    \Delta^{n-1}\times  \Delta^{n-1} \ar[d]^{d^n \times d^n}
    \ar[rr]^-{{  f_{d_n\sigma}\times g_{d_n\sigma}}} & &
D(c_{n-1})\times D(c_{n-1}) \ar[r]^-{+}\ar[d]^{D(\alpha_n)\times D(\alpha_n)} &
D(c_{n-1})\ar[d]^{D(\alpha_n)} \\
    \Delta^n \ar[r] &  
    \Delta^{n}\times  \Delta^{n} \ar[rr]^-{{  f_{\sigma}\times
      g_{\sigma}}} & & 
D(c_n) \times D(c_n) \ar[r]^-{+} &
D(c_n)
  }
$$
for which the first square obviously commutes, the second commutes
since $f$ and $g$ are in the prescribed subspace and the third commutes from the fact that $D(\alpha_n)$ is a
  group homomorphism. The inverse operation is similarly seen to be
  closed, hence the subspace defined above is a subgroup.
\qed
\end{proof}

\begin{proof}[Proof of Theorem \ref{thm:product}] 
  Let $L$ be an oriented link diagram with $c$ negative crossings.
  The space $\ETsp{k}(L)$ is constructed as follows. Let $I$ denote
  the category with objects $\{0,1\}$ and a single morphism from $0$
  to $1$, and $I^n$ the product of $I$ with itself $n$ times.  Let
  $\vect{0}$ be the initial object in $I^n$, and let $\ETQ$ be the
  result of adjoining one more object to $I^n$ and a single morphism
  from the new object to every object except $\vect{0}$. 

In \cite{ET-kh-spectrum} it is shown that there is a functor  $F\co \ETQ\to
  \AbGrp$ such that the $i\th$ derived functor of the inverse limit, ${\invlim_{\ETQ}}^i
  {F} $, is isomorphic to
  the $i\th$ unreduced Khovanov
  homology of $L$. The space $\ETsp{k}(L)$ is constructed by composing this functor with the
  Eilenberg-MacLane space functor  $K(-,k+c)$ and taking the homotopy
  limit of the resulting diagram of spaces. 

We may now use the explicit model for  Eilenberg-MacLane spaces given
by McCord. By applying Proposition \ref{prop:bifunctor} we define a diagram $D\colon \ETQ \ra \AbTop$ as the
composition
$$
\xymatrix{
\ETQ \ar[r]^F & \AbGrp \ar[rr]^{B(-,S^{k+c})} & &\AbTop .
}
$$
By the homotopy invariance property of the homotopy limit construction
we have
$$
\ETsp{k}(L) \simeq \holim_{\ETQ}{D}.
$$
By
  Proposition \ref{prop:holim}, the homotopy limit on the right is
  itself a topological abelian group, and hence, by 
  Proposition \ref{prop:TAG-is-K-pi-n}, a product of Eilenberg-MacLane
  spaces.
\qed
\end{proof}

\begin{corollary}
  The homotopy type of $\ETspace(L)$ is determined by $\Kh(L)$.
\end{corollary}

The spectrum $\KhSpace(L) =\{\Khsp{k}(L)\}$  constructed in~\cite{RS-khovanov} has the additional property that
the cellular cochain complex of the space $\Khsp{k}(L)$ is isomorphic to the Khovanov complex of $L$ (up to shift).
It follows from the description of the Khovanov homology of the mirror image (see \cite{Khovanov00}) that
$$
\wt{H}_i(\KhSpace(L))=\Kh^{-i}(-L)
$$
where $-L$ denotes the mirror of
$L$. The infinite symmetric product  $\Sym^\infty\Khsp{k}(L)$ is seen from the Dold-Thom theorem to be
$$
  \Sym^\infty\Khsp{k}(L) = \prod_n K(\wt{H}_n(\Khsp{k}(L) ),n)
$$
from which we have the following.
\begin{corollary}
  For large enough $k$, the space $\ETsp{k}(-L)$ is homotopy equivalent to the
  infinite symmetric product $\Sym^\infty\Khsp{k}(L)$.
\end{corollary}

We end by noting that the analogue or Theorem \ref{thm:product} for the spectra $\KhSpace(L)$ is not true. For all alternating knots $\KhSpace(L)$ is a wedge of Moore spaces~\cite{RS-khovanov}, however there are examples of non-alternating knots for which $\KhSpace(L)$ is not a wedge of Moore spaces (see \cite{LS2}).

\subsection*{Acknowledgements} We
 thank Tyler Lawson for several helpful
suggestions, including communicating Proposition~\ref{prop:holim} to us.

\bibliographystyle{amsalpha}
\bibliography{ELST}
\end{document}